\documentclass[graybox]{svmult}
\usepackage[varvw]{newtxmath}
\usepackage{newtxtext, framed, color}
\usepackage{type1cm}        % activate if the above 3 fonts are
\usepackage{graphicx}        % standard LaTeX graphics tool
\usepackage{multicol}        % used for the two-column index
\usepackage{hyperref}  %for hyperlinks
\newcommand{\hl}[1]{#1}
\newcommand{\T}{\mathbb{T}} %% torus
\newcommand{\C}{\mathbb{C}} %% complex numbers
\newcommand{\DD}{\mathbb{D}} %% unit disk
\newcommand{\ints}{\mathbb{Z}} %% integers
\newcommand{\R}{\mathbb{R}}
\newcommand{\cd}{\overline{\DD}} %% closed disk
\newcommand{\uhp}{\mathbb{H}}

\begin{document}

\institute{Washington University in St. Louis\\ Department of Mathematics\\ One Brookings Drive\\ St. Louis, MO 63130 USA\\ \email{geknese@wustl.edu}}


%\subjclass[2020]{ }

\title{Rational inner functions on the polydisk --- a survey}

\author{Greg Knese}

\thanks{Partially supported by NSF grant DMS \#2247702}

\date{\today}

\abstract{Rational inner functions are a generalization
of finite Blaschke products to several variables.
In this article we survey a variety of results about
rational inner functions 
related to interpolation, sums of squares formulas, 
and boundary behavior.  We mostly focus
on two variables however in the final section
we discuss higher dimensions.}

\maketitle

\tableofcontents

\section{Introduction}

\hl{Rational inner functions} are a multivariable generalization of finite
\hl{Blaschke products}.  Blaschke products are fundamental in one complex variable
especially in the case of operator related function theory
on the unit disk.  There has even been an entire book
written about them \cite{BBook}.
While rational inner functions in several variables
cannot possibly exhibit the centrality that finite Blaschke products
possess, they represent an exciting area of study overlapping
with the theory of stable polynomials, multivariable operator theory,
function spaces on domains in several variables, 
and real algebraic geometry.  We hope to illustrate this
in the present survey by reviewing connections to interpolation,
sums of squares formulas connected to operator theory,
and the study of boundary behavior of analytic functions. 
A great deal is known at this point, especially in two variables, 
but there remains a lot to be discovered.

Recall that a finite Blaschke product is a function $B$, analytic on the unit disk
$\DD = \{z \in \C: |z|<1\}$
of the form
\[
B(z) = \mu \prod_{j=1}^{n} \frac{z-a_j}{1-\bar{a}_j z}
\]
where $a_1,\dots, a_n \in \DD$ and $\mu$ belongs to the unit circle $\T = \{z \in \C: |z|=1\}$.
Finite Blaschke products are of course rational, analytic on $\DD$, and they are \emph{inner}, meaning
\[
|B(z)| = 1 \text{ for almost every } z \in \T.
\]
The use of ``almost every'' here is unnecessary because $B$ is analytic
on a neighborhood of the closed unit disk.  
This motivates our definition
of a rational inner function in several variables.  Our focus will be on
the domain 
\[
\DD^d = \{z= (z_1,\dots, z_d) \in \C^d: |z_1|,\dots, |z_d|<1\}
\]
called the (unit) \emph{polydisk} in $\C^d$
and its distinguished boundary
\[
\T^d =\{z = (z_1,\dots, z_d) \in \C^d: |z_1| =\dots = |z_d|=1\}
\]
called the $d$-torus.

\begin{definition}
Given polynomials $p,q \in \C[z_1,\dots, z_d]$ with no common factors, 
where $p(z) \ne 0$
for $z\in \DD^d$, the function $f = q/p$ is a \hl{\emph{rational inner function}}
if $|f(z)| = 1$ for almost every $z \in \T^d$.
\end{definition}

The first theorem in the subject is the following
characterization of rational inner functions.
 It helps
to define a natural reflection of a polynomial $p \in \C[z_1,\dots, z_d]$
to be
\[
\tilde{p}(z) = z^{n} \overline{p(1/\bar{z}_1,\dots, 1/\bar{z}_d)}
\]
where $n = \deg(p) := (\deg_{z_1}(p),\dots, \deg_{z_d}(p))$ is the multi-degree
of $p$.  We are using multi-index notation to define $z^n$.

\begin{theorem}\label{RIFchar}
Given $p,q \in \C[z_1,\dots, z_d]$ with no common factors, 
where $p(z) \ne 0$
for $z\in \DD^d$, the function $f = q/p$ is a \emph{rational inner function}
if and only if 
$q(z)$ is of the form
\[
a z^{m} \tilde{p}(z)
\]
where $a\in \T$, $m \in \ints_{\geq 0}^d$ is a $d$-tuple of 
non-negative integers.
\end{theorem}

The converse is a consequence of $|\tilde{p}| = |p|$ on $\T^d$.
Thus, every rational inner function is of the
form 
\[
z^m \frac{\tilde{p}(z)}{p(z)}
\]
for $p$ with no zeros in $\DD^d$ and
no factors in common with $\tilde{p}$ (the constant
$a$ can be absorbed into the polynomials).
We can also simply write the rational inner function as
$\frac{\tilde{p}}{p}$ by reflecting $p$ at a degree higher
than its multidegree.  In this way we have a simple uniform way of
writing rational inner functions. 
We shall give essentially Pfister's proof \cite{Pfister} in the next section. 
With Theorem \ref{RIFchar},
 it is easy to construct examples of rational inner functions
(that is, if one can easily construct polynomials with
no zeros on $\DD^d$.)

\begin{example}
For convenience, when working in two variables
we often write $(z,w)$ instead of $(z_1,z_2)$.
The polynomials $p_0(z,w) = 2-z - w$ and $p_1(z,w) = 3-z-w$
both have no zeros in $\DD^2$ and have no factors in common with
their reflections
\[
\tilde{p}_0(z,w) = 2zw - z- w \qquad \tilde{p}_1(z,w) = 3zw-z-w
\]
and therefore both lead to rational inner functions
\[
f_0(z,w) = \frac{2zw - z- w}{2-z-w} \quad f_1(z,w) = \frac{3zw-z-w}{3-z-w}.
\]
Since $p_1$ has no zeros on the closed bidisk $\cd^2$, $f_1$ extends
analytically to a neighborhood of $\cd^2$.
On the other hand, $p_0$ has a zero at $(1,1)$
and $f_0$ does not even extend continuously to $\cd^2$.
%In fact, one can show the cluster set of $f_0$
%at $(1,1)$ is the closed unit disk $\cd$. $\diamond$
%\iffalse
To see this, we first note that $f_0$ extends continuously
to $\cd^2\setminus \{(1,1)\}$ since $2-z-w = 0$ for 
$|z|,|w|\leq 1$ if and only if $z=w=1$.
Now, consider for each $a \in \R$, the limit
\[
\lim_{\theta \to 0} f_0(e^{i\theta}, e^{-i\theta + i a \theta^2}) 
= \lim_{\theta \to 0} \frac{ia \theta^2+ \theta^2 + O(\theta^3)}{-ia\theta^2 + \theta^2 + O(\theta^3)}
=
\frac{1+ia}{1-ia}
\]
which is an arbitrary element of $\T$ excluding $-1$. $\diamond$
\iffalse
while 
\[
f_0(e^{i\theta},1) \equiv f_0(1,e^{i\theta}) \equiv -1
\]
for $\theta \ne 0$.
Thus, the cluster set of $f_0$ restricted to $\T^2\setminus \{(1,1)\}
at $(1,1)$ contains $\T$.
\fi
\end{example}

Thus, the most basic difference between one and several variables
is that rational inner functions in two or more variables can
have boundary singularities.
The example $f_0$ above is often called the ``favorite RIF'' because it
is the simplest rational inner function with a boundary singularity
and it exhibits a variety of interesting phenomena.

Having introduced the concept of rational inner functions, 
why are we interested in them?  
First, polynomials with no zeros on the polydisk,
often called stable polynomials,
are of broad interest in a variety of areas of math
and rational inner functions offer a vehicle for studying
them.
Second, the \hl{Schur class} of the polydisk
consists of the analytic functions on the polydisk
bounded by one and rational inner functions
approximate Schur class functions locally
uniformly.  This makes rational inner functions
an attractive family of functions to use in tackling
a variety of problems because unlike with polynomials
the property of being bounded by one is assured.
An prominent example of this is the fact that one
can prove a sum of squares formula for rational
inner functions in two variables that
implies an operator inequality for bounded
analytic functions in two variables called \hl{And\^{o}'s inequality}
as well as Agler's Pick interpolation theorem
on the bidisk.  

Once rational inner functions have proven their worth 
we become open to studying them for their own sake.  
A lot of work has been done in recent years to understand
the level sets of rational inner functions and the behavior
of their singularities.  Some of this is motivated
by operator theory, some by pure function theory, 
and some by the natural quest for an understanding
of the behavior of the zero set of a stable polynomial
near a ``boundary'' zero.

\section{Basic facts}
We begin by establishing some basic facts
about rational inner functions.  
We hope it is worthwhile to the reader to 
see some details about the fundamentals---
we will be more sketchy about more modern 
developments.
The paper of Pfister \cite{Pfister} is probably the
first serious study of rational inner functions
and many of the results below are from this paper
as well as Rudin's book \cite{RudinBook}.

\begin{lemma} 
If $p \in \C[z_1,\dots, z_d]$ is nonzero, then
integrating with respect to normalized Lebesgue
measure we have
\[
\int_{\T^d} \log |p| > -\infty
\]
and in particular, the zero set of $p$ on $\T^d$
has measure zero.
\end{lemma}

\begin{proof}
One can prove an iterated Jensen-Poisson inequality 
\[
\int_{\T^d} \log|p(z)| P_a(z) d\sigma(z) \geq \log|p(a)|
\]
where $d\sigma$ is normalized Lebesgue measure, $a \in \DD^d$, 
$P_a(z) = \prod_{j=1}^{d} \frac{1-|a_j|^2}{|1-\bar{a}_j z_j|^2}$
is the Poisson kernel.
Then, since $p$ is not identically zero we can find $a\in \DD^d$
such that $p(a)\ne 0$ yielding
\[
\int_{\T^d} \log|p(z)| P_a(z) d\sigma(z) > -\infty
\]
and hence $\int_{\T^d} \log|p| >-\infty$. \qed
\end{proof}

\begin{proof}[Proof of Theorem \ref{RIFchar}]
Given $f= q/p$ with $|q/p|=1$ almost everywhere on $\T^d$
we have $|q|=|p|$ almost everywhere on $\T^d$.
Let $n = \deg p, m= \deg q$.
Then,
\[
z^n q(z) \tilde{q}(z) = z^m p(z) \tilde{p}(z)
\]
almost everywhere on $\T^d$.
However, two polynomials that agree almost everywhere
on $\T^d$ are necessarily identical on $\C^d$ (for instance, 
their Fourier coefficients are the polynomials coefficients).
Since $p$ and $q$ are assumed to have no 
common factors we must have $p(z)$
divides $z^n \tilde{q}(z)$.  Writing
$z^n \tilde{q}(z) = p(z) g(z)$
for some factor $g$, we have 
\[
|z^n \tilde{q}(z)| = |q(z)| = |p(z)| = |p(z)||g(z)|
\]
on $\T^d$.  Therefore, $|g(z)| = 1$ almost everywhere.
Also, since $p$
has multidegree $n$, the multidegree of $g$
is the same as $\tilde{q}$.
Set $k = \deg \tilde{q} \leq \deg q$.    
We have $\tilde{g}(z) g(z) = z^k$ on $\C^d$. 
Since $g(z)$ has multidegree $k$, we see that $g(z)$ is a multiple of $z^k$,
say $g(z) = a z^k$. We necessarily have $|a| = 1$.  
Therefore, $z^n \tilde{q}(z) = a z^k p(z)$.
Reflecting at degree $m+n$ we have
$q(z) = \bar{a} z^{m-k} \tilde{p}(z)$,
concluding the proof. \qed
\end{proof}

\begin{proposition}
Rational inner functions are bounded by $1$ in $\DD^d$.
\end{proposition}

\begin{proof}
It is enough to show $|\tilde{p}(z)| \leq |p(z)|$ in $\DD^d$.
If $p$ has no zeros on $\cd^d$,
this is immediate by the maximum principle applied to $\tilde{p}/p$.
Otherwise, we consider $0<r<1$ and $p_r(z) = p(rz)$ for which
$\widetilde{p_r}(z) = z^n \overline{p(r/\bar{z})} = r^{|n|} \tilde{p}(z/r)$
and
\[
|r^{|n|}\tilde{p}(z/r)| \leq |p(rz)|.
\]
Sending $r\to 1$ we see $|\tilde{p}(z)| \leq |p(z)|$. \qed
\end{proof}

The following theorem was proven by Carath\'eodory in one variable;
the several variable version was probably first stated and proved by Pfister \cite{Pfister}; 
see also \cite{RudinBook} (Theorem 5.5.1).

\begin{proposition}\label{carathm}
Let $f:\DD^d \to \C$ be analytic and $|f|\leq 1$.
Then, for any $r<1$ and $\epsilon>0$, there exists a rational inner function $\phi$
such that $|f(z) - \phi(z)| <\epsilon$ for $z \in r\cd^d$.
\end{proposition}

\begin{lemma}\label{polyapprox}
Let $f:\DD^d \to \C$ be analytic and $|f|\leq 1$.
Then, for any $r<1$ and $\epsilon>0$, there exists a polynomial $P$
such that $|f(z) - P(z)| <\epsilon$ for $z \in r\cd^d$ and $|P(z)|<1$ for $z \in \cd^d$.
\end{lemma}
\begin{proof}
For each $r<1$, write $f_r(z) = f(rz)$.  
For fixed $r<1$ there exists $s<1$ such that $|f_r(z)-f_{rs}(z)|<\epsilon/2$
for $z \in \cd^d$ since $f_r$ is uniformly continuous on $\cd^d$.  
Note that $\sup_{\cd^d} |f_s| = A<1$ by the maximum principle.
Choose a Taylor polynomial $P$ of $f_s$ so that $|f_s(z)-P(z)|<\min(1-A, \epsilon/2)$
for $z \in \cd^d$.  Then, $|P(z)| <1$ for $z \in \cd^d$ and
\[
|f_r(z)- P_r(z)| \leq |f_r(z)-f_{rs}(z)| + |f_{rs}(z) - P_r(z)| < \epsilon \text{ for } z \in \cd^d. \text{ \qed}
\]
 \end{proof}

\begin{proof}[Proof of Proposition \ref{carathm}]
Choose $P$ as in Lemma \ref{polyapprox}. 
It is enough to show that we can approximate $P$ locally
uniformly by a rational inner function.
Let $n = \deg P$ be the multidegree of $P$.
Since $|P|=|\tilde{P}|$ on $\T^d$, the maximum principle
implies $|\tilde{P}(z)| <1$ on $\cd^d$.
Let $N \in \mathbb{Z}^d_{>0}$.  
Then, $1+z^N \tilde{P}(z)$ has no zeros in $\cd^d$
and multidegree $n+N$.
Reflecting with respect to degree $n+N$ we see
\[
\phi(z) = \frac{P(z) + z^{N+n}}{1+ z^N \tilde{P}(z)}
\]
is a rational inner function.
Note
\[
\phi(z) = P(z) + z^N \frac{z^n-\tilde{P}P}{1+z^N \tilde{P}}
\]
and we can make the right-most expression arbitrarily small
if one of the components of $N$ is large enough since $z \in r\cd^d$.\qed
\end{proof}

\section{Interpolation} \label{interp}

The classical \hl{Pick interpolation theorem} gives necessary
and sufficient conditions when the following interpolation problem
can be solved: given $z_1,\dots, z_n\in \DD$ and $t_1,\dots, t_n \in \DD$
when does there exist analytic $f:\DD\to \DD$ with $|f|\leq 1$
such that $f(z_j) = t_j$ for $j=1,\dots, n$?
The answer is if and only if 
\begin{equation}\label{pickmatrix}
\left( \frac{1-t_j\bar{t}_k}{1-z_j \bar{z}_k} \right)_{j,k} \text{ is positive semi-definite.} 
\end{equation}
This theorem can be addressed entirely using
rational inner functions (in this case, finite Blaschke products).
Indeed, it is possible to show \eqref{pickmatrix}
implies that we have a finite Blaschke product
that interpolates,
and conversely, by the approximation result, Proposition \ref{carathm},
we only need to verify necessity for rational inner functions.

Very briefly, one way to prove sufficiency is to factor 
\[
\frac{1-t_j\bar{t}_k}{1-z_j \bar{z}_k}  = \langle u_j, u_k\rangle
\]
for some vectors $u_j$ in $\C^N$ (where $N$ is the rank of the positive semi-definite matrix).
After a process called a ``lurking isometry'' one can find a function of the form
\[
\phi(z) =A + z B(I-zD)^{-1}C
\]
that interpolates.
Here $U = \begin{pmatrix} A & B\\ C & D \end{pmatrix}$ is a block unitary matrix
with $A$ a $1\times 1$ matrix and $D$ an $N\times N$ matrix.
It turns out functions of this form are finite Blaschke products.
The lurking isometry argument is given for instance in \cite{PickBook}.

To prove necessity of \eqref{pickmatrix} for finite Blaschke products, 
one approach uses ideas from \hl{orthogonal polynomials}.
Given a polynomial $p\in\C[z]$ with
no zeros on $\cd$ and degree $n$, 
we consider the measure $d\mu = \frac{1}{|p|^2} d\sigma$
on $\T$.  
One can show that in $L^2(\mu)$, $p \perp z,\dots, z^n$
and $\tilde{p} \perp 1,\dots, z^{n-1}$.
If $q_0,\dots, q_{n-1}$ is an orthonormal basis for $\text{span}\{z^j:j=0,\dots,n-1\}$
in $L^2(\mu)$, then
$q_0,\dots, q_{n-1}, \tilde{p}$ and
$p, zq_0,\dots, zq_{n-1}$
are both orthonormal bases for $\text{span}\{z^j:j=0,\dots,n\}$.
So, there exists a unitary matrix $U$ such that
\[
U \begin{pmatrix} \tilde{p}(z) \\ q_0(z) \\ \vdots \\ q_{n-1} \end{pmatrix} = 
\begin{pmatrix} p(z) \\ z q_0(z) \\ \vdots \\ z q_{n-1}(z) \end{pmatrix}.
\]
Unitarity implies that
\[
\tilde{p}(z) \overline{\tilde{p}(w)} + \sum_{j=0}^{n-1}q_j(z) \overline{q_j(w)}
=
p(z) \overline{p(w)} + z\bar{w} \sum_{j=0}^{n-1}q_j(z) \overline{q_j(w)}
\]
and this rearranges into
\[
\frac{p(z) \overline{p(w)} - \tilde{p}(z) \overline{\tilde{p}(w)} }{1-z\bar{w}} = \sum_{j=0}^{n-1}q_j(z) \overline{q_j(w)}.
\]
Finally, dividing by $p(z) \overline{p(w)}$ yields
\[
\frac{ 1- \phi(z) \overline{\phi(w)}}{1-z\bar{w}} = \sum_{j=0}^{n-1}\frac{q_j(z)}{p(z)} 
\frac{\overline{q_j(w)}}{\overline{p(w)}}.
\]
In particular, we immediately get \eqref{pickmatrix} for $t_j = \phi(z_j)$.
(Note we can get the most general finite Blaschke product 
by allowing $\tilde{p}$ to have a monomial factor in front.)
This method at first seems ad hoc but it has a variety of other purposes
as the measures $\frac{1}{|p|^2}d\sigma$, called \hl{Bernstein-Szeg\H{o} measures},
can be used to solve a different type of interpolation problem, namely
the truncated \hl{trigonometric moment problem}; see \cite{Landau}.

Everything above can be generalized to two variables if one is careful
about how one generalizes.

\begin{theorem}[Agler's Pick interpolation theorem] \label{aglerpick}
Given $(z_1,w_1),\dots, (z_N,w_N) \in \DD^2$ and $t_1,\dots, t_N \in \DD$,
there exists an analytic function $f:\DD^2 \to \DD$ with $f(z_j,w_j) = t_j$ 
for $j=1,\dots, N$ if and only if 
there exist positive semi-definite matrices $A, B$ such that
\begin{equation} \label{aglerdecomp}
1- t_j\bar{t}_k = (1-z_j\bar{z}_k) A_{j,k} + (1-w_j\bar{w}_k) B_{j,k}.
\end{equation}
\end{theorem}

Just as in one variable, we can prove sufficiency by solving
for a rational inner function that solves the interpolation problem
and it is enough to prove necessity for rational inner functions.
Addressing sufficiency, 
one can show that if a formula like \eqref{aglerdecomp} holds
then the interpolation problem can be solved with a function
of the form
\begin{equation} \label{tfr}
 \phi(z,w) =A + B \Delta(z,w) (I-D \Delta(z,w))^{-1}C
\end{equation}
where we again have a block $(1+N)\times (1+N)$ unitary $U =  \begin{pmatrix} A & B \\ C & D \end{pmatrix}$ 
but 
\begin{equation} \label{delta}
\Delta(z,w) = \begin{pmatrix} z I_{n} & 0 \\ 0 & w I_{m} \end{pmatrix}
\end{equation}
where $n+m = N$.  This function is a rational inner function in two variables.

Conversely, the original proof of necessity in this theorem 
used a cone separation argument
and \hl{And\^{o}'s inequality} from operator theory; see \cite{PickBook, AM99, Agler1}.
It was perhaps the insight of Cole-Wermer \cite{CW99} that
it is enough to prove a \hl{sum of squares} formula of the form
\begin{equation} \label{sos}
|p(z,w)|^2 - |\tilde{p}(z,w)|^2 = (1-|z|^2) \sum_{j=1}^{n} |E_j(z,w)|^2 + (1-|w|^2) \sum_{j=1}^{m} |F_j(z,w)|^2
\end{equation}
for $p \in \C[z,w]$ with no zeros in $\DD^2$, and $E_j,F_j \in \C[z,w]$.
The moment problem approach mentioned above can be used to
establish this formula independent of operator theoretic tools---this was
done in Geronimo-Woerdeman \cite{GW} and we discussed this connection
further in \cite{gKBS}.
It turns out that formulas like \eqref{tfr} are interchangeable with
sums of squares formulas as in \eqref{sos}.  
This equivalence is explained in detail in \cite{gKKummert}.
Kummert \cite{Kummert} proved that all rational inner functions in two variables
satisfy a formula like \eqref{tfr} using an approach motivated by
control engineering, and thus gave an approach to
much of this material that predates much of the mathematics literature.  
We gave a mathematical exposition of Kummert's work in \cite{gKKummert}.

It turns out to be quite interesting to develop and refine the formula
\eqref{sos} further.  It will allow us to prove things about polynomials
with no zeros on $\DD^2$ and lead to additional new questions that we discuss 
in the next section.

For now, let us remark that currently there is no known interpolation 
theorem for $\DD^d$ when $d>2$.  
An `interpolation theorem' would mean that we reduce 
interpolation for bounded analytic functions to
searching over a finite dimensional set.  Theorem \ref{aglerpick}
does this even if the conditions given could be difficult to check.
The following is not even known:

\begin{problem} \label{ratinterp}
For $d>2$, given $z_1,\dots, z_N\in \DD^d$ and $f:\DD^d \to \DD$ analytic, does there
exist a rational inner function $\phi$ in $d$ variables with
$f(z_j) = \phi(z_j)$  for $j=1,\dots, N$?
\end{problem}

The answer to this question is affirmative if $N\leq 3$ by work in \cite{Kosinski}, \cite{gK3x3}.

\section{Sums of squares} % and determinantal representations}

Using either a moment problem approach (\cite{GW}, \cite{gKBS})
or Kummert's approach (\cite{Kummert}, \cite{gKKummert}),
it is possible to prove the following refined version of \eqref{sos}.

\begin{theorem} \label{sosthm}
Let $\phi(z,w) = \frac{\tilde{p}(z,w)}{p(z,w)}$ be a rational inner function,
where $p \in \C[z,w]$ has no zeros in $\DD^2$, no factors in common
with $\tilde{p}(z,w)$, and $\tilde{p}$ has bidegree $(n,m)$. 
Then,
there exist polynomials $E_1,\dots, E_n$ of bidegree at most $(n-1,m)$ and $F_1,\dots, F_m$ 
of bidegree at most $(n,m-1)$ 
such that
\[
|p(z,w)|^2 - |\tilde{p}(z,w)|^2 = (1-|z|^2) \sum_{j=1}^{n} |E_j(z,w)|^2 + (1-|w|^2) \sum_{j=1}^{m} |F_j(z,w)|^2.
\]
Equivalently, there exists a $(1+n+m)\times (1+n+m)$ unitary $U$ such that \eqref{tfr} holds.
\end{theorem}

The key thing is that the numbers of squares are precisely related to 
the bidegree of $\tilde{p}$.  
This precision is important in one of the main theorems
of \cite{AMY} on generalizing \hl{Loewner's theorem} on matrix monotone
functions to several variables.
It is also crucial to the following corollary, which gives a 
straightforward way to construct all polynomials
with no zeros in $\DD^2$. 

\begin{corollary}\label{detrep}
Suppose $p \in \C[z,w]$ has no zeros in $\DD^2$ and bidegree $(n,m)$.
Then, there exists an $(n+m)\times (n+m)$ contractive matrix $D$
such that
\[
p(z,w) = p(0,0) \det(I - D \Delta(z,w))
\]
with $\Delta$ as in \eqref{delta}.
\end{corollary}

If $p$ has no factors in common with $\tilde{p}$, 
the matrix $D$ can be given with additional structure---
it can be chosen to be a rank one perturbation of a unitary;
see Section 9 of \cite{gKextreme}.
In fact, polynomials with no zeros in $\DD^2$ can be
factored into $p = p_1p_2$ where
$p_1$ has no factors in common with $\tilde{p}_1$
while $p_2$ is a constant multiple of $\tilde{p}_2$.
In the latter case we can find a determinantal representation
for $p_2$ where $D$ is in fact a unitary.
Note that a direct proof of Corollary \ref{detrep} is 
given in \cite{KummertStable}, \cite{DrexelStable}.

Returning to Theorem \ref{sosthm}, it is worth
pointing out that such sums of squares formulas are in 
general not unique.  
A trivial way they are not unique is that we could replace
\[
\begin{pmatrix} E_1 \\ \vdots \\ E_n\end{pmatrix} \text{ and } \begin{pmatrix} F_1 \\ \vdots \\ F_m\end{pmatrix}
\]
with unitary multiples.  The sums of squares terms can
be non-unique in a non-trivial way as well.

\begin{example}
For $p_1 = 3-z-w$ there exist distinct real numbers $A,B$ 
(one can solve for them) so that
\[
\begin{aligned}
|p_1(z,w)|^2 - |\tilde{p}_1(z,w)|^2
&= (1-|z|^2)|A+Bw|^2 + (1-|w|^2)|B+Az|^2 \\
&= (1-|z|^2)|B+Aw|^2 + (1-|w|^2)|A+Bz|^2.
\end{aligned}
\]
Thus, this example has two inequivalent sums of squares formulas. $\diamond$
\end{example}

Understanding unique sums of squares decompositions or describing all such decompositions
turns out to be surprisingly deep.  It forces us to confront polynomials
with boundary zeros.
The following was proven in \cite{gKnoz} (Theorem 1.15), \cite{gKintreg} (Corollary 13.6).

\begin{theorem}
Let $p \in \C[z,w]$ have no zeros in $\DD^2$, no factors in common with
$\tilde{p}$, and $\deg \tilde{p} = (n,m)$.
Consider the sums of squares terms $\mathcal{E}(z,w) = \sum_{j=1}^{n}|E_j(z,w)|^2$, 
$\mathcal{F}(z,w) = \sum_{j=1}^{m}|F_j(z,w)|^2$ appearing in Theorem \ref{sosthm}.
Then, the following are equivalent
\begin{enumerate}
\item  the sums of squares terms $\mathcal{E},\mathcal{F}$ are unique
(up to unitary transformation of the sums of squares)
\item  the sums of squares terms possess the symmetries
\[
\mathcal{E}(z,w) = |z^{n-1}w^{m}|^2\mathcal{E}(1/\bar{z},1/\bar{w}) \quad 
\mathcal{F}(z,w) = |z^{n}w^{m-1}|^2\mathcal{F}(1/\bar{z},1/\bar{w})
\]

\item If $f \in \C[z,w]$ has $\deg f \leq (n-1,m-1)$ and $\int_{\T^2} \left|\frac{f}{p}\right|^2 d\sigma <\infty$
then $f =0$.

\item $p$ and $\tilde{p}$ have $2nm$ common zeros on $\T^2$ counted using
intersection multiplicities.
\end{enumerate}

\end{theorem}

Item 2 above is related to the fact that every $p$
has two special sums of squares decompositions.
Loosely speaking, one such decomposition 
has the property that $\mathcal{E}$ is maximal among decompositions
and the other has $\mathcal{F}$ maximal.  Item 2 then
gives a condition that makes these two special decompositions coincide---and hence
all decompositions coincide. 
See Theorem 1.3 of \cite{gKnoz}.

The third and fourth items express the idea that $p$ has
the maximal number of zeros on $\T^2$ without having
a curve of zeros on $\T^2$.  We called such polynomials 
\emph{saturated} in \cite{gKextreme}.  The presence
of the $L^2$ condition in item 3 is related to the
construction of the sums of squares terms
in \cite{gKnoz} using bases for certain subspaces of 
$L^2(\frac{1}{|p|^2} d\sigma)$.  
Regarding item 4, \hl{B\'ezout's theorem} for $\C_{\infty}\times \C_{\infty}$
would say that $p$ and $\tilde{p}$ have $2nm$ common 
zeros globally, and item 4 is the condition that all of these zeros
occur on $\T^2$.  Again, common zeros are counted using
intersection multiplicity as in elementary algebraic curve theory as in \cite{Fulton}.

\begin{example}
The polynomial $p_0 = 2-z-w$ has unique sums of squares decomposition
\[
|p_0(z,w)|^2 - |\tilde{p}_0(z,w)|^2 = (1-|z|^2)2|1-w|^2 + (1-|w|^2)2|1-z|^2.
\]
This can be checked using any of the above conditions 2-4 but symmetry
is probably easiest to use.  $\diamond$
\end{example}

It is probably not feasible to characterize
all sums of squares decompositions, but
it turns out it is possible to characterize
the decompositions where the number of
squares, $n$ for $\mathcal{E}$ and $m$ 
for $\mathcal{F}$, are minimal as in Theorem \ref{sosthm}.
The characterization is too technical to state here (see Theorem 9.3 of \cite{gKintreg})
but the most interesting aspect is that the minimal decompositions
are in one-to-one correspondence with the invariant subspaces
of a special pair of commuting contractive matrices.

Assume $p$ is given as in Theorem \ref{sosthm}.
Let 
\[
\mathcal{G} = \left\{q \in \C[z,w]\cap L^2(\frac{1}{|p|^2}d\sigma): \deg q \leq (n-1,m-1)\right\}.
\]
Let $P$ denote orthogonal projection onto $\mathcal{G}$
in $L^2(\frac{1}{|p|^2}d\sigma)$.  
Consider the operators $T_z, T_w: \mathcal{G} \to \mathcal{G}$ given by
\[
T_z q = P(z q) \qquad T_w q = P(w q).
\]
It is a non-trivial fact that $T_z$ and $T_w^*$ commute
and the minimal decompositions of $p$ 
are in correspondence with the joint invariant
subspaces of $(T_z, T_w^*)$.
This result has important antecedents in work of 
Ball-Sadosky-Vinnikov \cite{BSV} and Geronimo-Woerdeman \cite{GW}.

Understanding sums of squares formulas for
rational inner functions naturally leads to several questions.  
First, can we better understand the space
\[
\C[z,w] \cap L^2\left(\frac{1}{|p|^2} d\sigma\right) \text{ ?}
\]
Second, what happens in higher dimensions?
The first question forces us to confront a local
theory of stable polynomials (those with no zeros
on $\DD^d$) near a boundary zero.
Along the way we learn much about the geometry
of level sets of rational inner functions on $\T^2$.
We explore this in the next section.
The second question is inextricably linked
to operator theory (the failure of \hl{von Neumann's
inequality} in 3 or more variables). 
We explore this in the final section.

\section{Boundary behavior}

In this section we are interested in the behavior
of rational functions $q/p$ near a boundary zero.
A variety of behaviors could be studied including 
existence of non-tangential boundary values (or 
derivatives) as well as boundedness or integrability
near/on the boundary.  

Non-tangential behavior of general bounded analytic
functions on polydisks has been studied a great 
deal in recent years.  See \cite{Abate}, \cite{AMYcara}, \cite{AMYbook}, \cite{MPcara}, 
\cite{ATYboundary}, \cite{Pascoe-cara}, \cite{Tully-Doyle}.  
However, there are some notable simplifications that occur 
for rational and rational inner functions.
A main theorem of \cite{gKintreg}, \cite{BKPS} is the following.

\begin{theorem} \label{nontan}
Suppose $p,q\in \C[z_1,\dots, z_d]$, $p$ has no zeros on $\DD^d$,
and $f=q/p$ is bounded on $\DD^d$.  Then, $f$ has a non-tangential
limit at every point of $\T^d$.  
\end{theorem}

A \hl{non-tangential limit} here just means that we approach a boundary point
in such a way that the distance to the point is comparable to the distance
to the boundary.  A first step in proving this and related local problems is to 
switch to the setting of the poly-upper half plane 
\[
\uhp^d = \{x=(x_1,\dots, x_d) \in \C^d: \Im x_1,\dots, \Im x_d>0\}
\]
via conformal mapping where a boundary singularity is sent
to the origin.  
For instance, the polynomial $p_0(z,w) = 2-z-w$ converts to
\[
(1-ix)(1-iy) p_0\left(\frac{1+ix}{1-ix}, \frac{1+iy}{1-iy}\right) = 
-2i( x+y - 2ixy).
\]
Reflection in this setting becomes conjugation of coefficients:
\[
\bar{p}(x,y) = \overline{p(\bar{x},\bar{y})}.
\]
Thus, the favorite rational inner function $\frac{2zw-z-w}{2-z-w}$ becomes (after disregarding constants in front)
\begin{equation} \label{favuhp}
\frac{x+y+2ixy}{x+y-2ixy}.
\end{equation}
The key local result needed for proving Theorem \ref{nontan}
is the following:

\begin{theorem}
Suppose $p \in \C[x_1,\dots, x_d]$ has no zeros in $\uhp^d$ and $p(0)=0$.
Write out the homogeneous expansion of $p(x) = \sum_{j=M}^{N} P_j(x)$
where $P_j$ is homogeneous of degree $j$ and $P_M \not\equiv 0$.
Writing $P_{M+1} = A_{M+1} + i B_{M+1}$ for $A_{M+1}, B_{M+1} \in \R[x_1,\dots, x_d]$ we have the following.
\begin{enumerate}
\item $P_M$ has no zeros in $\uhp^d$ and there exists $\mu \in \T$ such that $\mu P_M \in \R[x_1,\dots, x_d]$.
\item Assuming $\mu=1$ in part (a) and $p$ has no factors in common 
with $\bar{p}$, we have $B_{M+1} \not\equiv 0$ and $\frac{A_M}{B_{M+1}}$ maps $\uhp^d$ into $\uhp$.
\end{enumerate}
\end{theorem}

This is stated as the ``Homogeneous Expansion'' theorem in \cite{BKPS} where
the first part is attributed to \cite{ABG}.  The next step is proving that a rational function
with given homogeneous expansions 
\begin{equation} \label{homograt}
\frac{\sum_{j=K}^{L} Q_{j}(x)}{\sum_{j=M}^{N} P_j(x)}
\end{equation}
has a non-tangential limit at $0$ if and only if $K=M$ and $Q_K$ is a multiple of $P_M$.
This property is evident in the example \eqref{favuhp} where $M=1$ and $P_1 = x+y$.
Finally, one proves that this property holds for bounded rational functions.

Next, we consider boundedness as well as boundary integrability near a boundary point. 
Boundedness is interesting because of the natural question of simply characterizing 
the bounded rational functions.  Square integrability is of interest because of the previous section
where certain sums of squares formulas depend on spaces of square integrable rational functions.
Based on the above, one might hope that boundedness near a boundary point (meaning
boundedness on $\uhp^d$ intersected with a neighborhood of $0$)
would simply depend on the condition that $Q_M = c P_M$ as in \eqref{homograt}.
Unfortunately, this is only true in the simplest examples like \eqref{favuhp}.
\begin{example}
The polynomial
\[
p_2(x,y) = x+y - ix^2 -ixy - 2x^3-6x^2y + 4ix^3y  
\]
has no zeros in $\uhp^2$
while 
\[
\frac{x+y}{p_2(x,y)}
\]
is unbounded near $(0,0)$.  
Indeed, we can show it is unbounded on $\R^2\setminus\{(0,0)\}$
near $(0,0)$
and then it will be unbounded on $\uhp^2$ near $0$
by continuity.
Setting $y = -x -4 x^3$ we have
\[
\frac{-4x^3}{24x^5-16i x^6} = \frac{1}{-6x^2 + 4ix^3}
\]
which goes to $\infty$ as $x\to 0$.
This is Example 5.6 of \cite{BKPS}. $\diamond$
\end{example}
 
To understand what is going on we look at the \hl{Puiseux series}
for the branches of $p(x,y)$ with no zeros in $\uhp^2$ but $p(0,0)=0$. 
Recall that by the \hl{Newton-Puiseux theorem}, the zero set of $p$
near $(0,0)$ can be parametrized (injectively) by maps of the form
\[
t\mapsto (t^n, \phi(t))
\]
where $n$ is a positive integer and $\phi$ is analytic in a neighborhood of $0 \in \C$
with $\phi(0) = 0$.  
This leads to a factorization of $p$ with factors of the form 
\[
\prod_{j=1}^{n} (y - \phi(\mu^j x^{1/n}))
\]
where $\mu = \exp(2\pi i/n)$.   
(See \cite{simonbasic}, Theorems 3.5.1, 3.5.2.)
In the current case, we have the property that whenever $t^n \in \uhp$
then $\phi(t) \notin \uhp$ (i.e. $\phi(t) \in -\overline{\uhp}$ the closed lower half plane).  
It turns out that we can characterize this.
It helps to replace $\phi$ with its negative, so that our characterization has positive parameters.

\begin{theorem} \label{localparam}
Suppose $\phi(t)$ is analytic in a neighborhood of $0\in \C$ and $\phi(0)=0$.
Let $n$ be a positive integer and suppose that 
\[
t\mapsto (t^n, -\phi(t))
\]
is an injection into $\C^2 \setminus \uhp^2$.
Then, there are two cases:
\begin{enumerate}
\item (real stable type) $n=1$, $\phi$ has real coefficients and $\phi'(0)>0$.

\item (pure stable type) There exist a positive integer $L$, $q\in \R[t]$ with $\deg q <2L$, $q(0)=0$, $q'(0)>0$, and $\psi(t)$
analytic in a neighborhood of $0$ with $\Im(\psi(0)) > 0$ such that
\[
\phi(t) = q(t^n) + t^{2Ln} \psi(t).
\]
\end{enumerate}
\end{theorem}

In the first case, $p(x,y)$ will have an analytic factor $y + \phi(x)$ with $\phi'(0)>0$ where
$\phi$ has real coefficients.  This implies $p$ has a smooth/analytic
 curve of zeros in the distinguished boundary $\R^2$.
In the second case, $p(x,y)$ will have a factor of the form
\[
\prod_{j=1}^{n}(y + q(x) + x^{2L}\psi(\mu^j x^{1/n})).
\]
We will prove the theorem here because its original statement (as Lemma C.3 of \cite{gKintreg}) 
neglected the real coefficient case (as it was not relevant in that context).

\begin{proof}
Let us write out the initial term of $\phi(t) = a t^r + \dots$ with $a\ne 0$.
We first show that $r$ equals $n$ and $a>0$.
Writing $t = |t| e^{i\theta}$, we have that for $n\theta \in (0,\pi) + 2\pi \ints$, $\phi(|t|e^{i\theta}) \in \overline{\uhp}$.
Writing $a = |a| e^{i\alpha}$ we have
\[
\lim_{|t|\searrow 0} \frac{\phi(|t|e^{i\theta})}{ |t|^r} = |a| e^{i\alpha}e^{ir\theta}
\]
and therefore $n\theta \in (0,\pi) + 2\pi \ints$ implies $\alpha + r\theta \in [0,\pi] + 2 \pi \ints$.
%On the other hand, if $e^{i(\alpha+r\theta)}$ belongs to $\uhp$ then for $t$ small
%enough $\phi(t) \in \uhp$ and then we must have $t^n \notin \uhp$.
%This means $\alpha+r\theta \in (0,\pi) + 2\pi \ints$ implies $n\theta \in [\pi,2\pi] + 2\pi \ints$.

This means that intervals of the form 
$(\alpha + (2j) \frac{\pi r}{n}, \alpha + (2j+1) \frac{\pi r}{n})$  for $j\in \ints$
must be contained in $[0,\pi] + 2\pi \ints$.  
Since the width of the open intervals is $r\pi/n$ we must have $r \pi/n \leq \pi$ so that $r\leq n$.
The contrapositive says if $\alpha+r\theta \in (\pi,2\pi) + 2\pi \ints$ then $n\theta \in [\pi,2\pi] + 2 \pi \ints$
and a similar argument shows $n\leq r$.  In order
to have $(\alpha, \alpha + \pi) \subset [0,\pi]+2\pi \ints$ 
we must have $\alpha$ an even multiple of $\pi$.  Therefore, $a = |a| e^{\alpha} > 0$.

Next we note that by continuity we must have that whenever $t^n \in \R$ then $\phi(t) \in \overline{\uhp}$.
Writing $\phi(t) = \sum_{j=n}^{\infty} a_j t^j$, we define
\[
\phi_0(t) = \sum_{j=1}^{\infty} \Re(a_{jn}) t^{jn} = a t^n + \Re(a_{2n})t^{2n} + \cdots
\]
which just extracts the real coefficients of powers $t^{jn}$.
If $\phi(t) = \phi_0(t)$, then $t\mapsto (t^n,\phi(t))$ is not
injective unless $n=1$.  Thus, in this case $\phi$ has the first form of the theorem.

So, suppose $\phi(t) \ne \phi_0(t)$.  Since $t^n \in \R$ implies $\phi_0(t) \in \R$,
the function 
\[
\tilde{\phi}(t) = \phi(t) - \phi_0(t)
\]
has the property that $t^n \in \R$ implies $\tilde{\phi}(t) \in \overline{\uhp}$.
Let us write 
$\tilde{\phi}(t) = \sum_{j=m}^{\infty} b_j t^j =  b t^{m} + \cdots$, where $b=b_m = |b| e^{i\beta} \ne 0$ and $m\geq 2$.  
Writing $t = |t|e^{i\theta}$ and taking a limit as before we see that $n\theta \in \pi \ints$ implies 
$\beta + m\theta \in [0,\pi] + 2\pi \ints$.  
In other words, after dividing by $\pi$
\[
\beta/\pi + \frac{m}{n} \ints \subset [0,1] + 2\ints.
\]
By an elementary lemma in \cite{gKint} (Lemma 3.4),
this implies $m/n$ is an integer and if $m/n$ is odd then $\beta/\pi \in \{0,1\}$.
The latter condition would mean the first term of $\tilde{\phi}(t)$ has the form $b t^{(2L+1)n}$
with $b \in \R$.  But $\phi_0$ removed all such terms, so we must have
$m =2Ln$ for some $L$.  In this case, $\tilde{\phi}(t)$ begins $b t^{2Ln}$ and $b$ must 
be purely imaginary (and nonzero) since $\phi_0$ removed the real parts of such terms.
Since $\beta \in [0,\pi]$, we must have $\beta = \pi/2$; namely $b$ is a positive multiple of $i$.
Thus, 
\[
\phi(t) = \phi_0(t) + \tilde{\phi}(t)
\]
which we can write as $q(t^n) + t^{2Ln} \psi(t)$ 
with $q$ and $\psi$ as in the theorem statement. \qed
\end{proof}

\begin{example}
The polynomial $p = x+y-2ixy$ factors as 
\[
(1-2ix)\left(y + \frac{x}{1-2ix}\right) = (1-2ix)(y + x + 2ix^2 + \text{ higher order})
\]
and therefore $\phi(x) = \frac{x}{1-2ix} = x + x^2 \frac{2i}{1-2ix}$ is an example of the
theorem with $q(x)=x$, $L=1$, $n=1$, $\psi(x) = \frac{2i}{1-2ix}$.  
More complicated examples are presented in \cite{BKPS}.$\diamond$
\end{example}

If $p(x,y)$ has no zeros in $\uhp^2$ and no factors in common with $\bar{p}$,
then all of $p$'s branches are of the second type called pure stable type.
This is because branches of the first type (real stable type) have infinitely many
common zeros on $\R^2$ and this would imply $p$ and $\bar{p}$ have a common
factor by B\'ezout's theorem.
We obtain the following complete local description of such $p$.
We use the notation $\C\{x,y\}$ to denote the set of power series
in $x,y$ that converge in a neighborhood of the origin.

\begin{theorem}[Theorem 2.16 in \cite{BKPS}] \label{purestable}
Suppose $p\in \C[x,y]$ has no zeros in $\uhp^2$, no factors in 
common with $\bar{p}$, and $p(0,0)=0$.
Then, there exist $u \in \C\{x,y\}$ with $u(0,0)\ne 0$,
positive integers $2L_1,\dots, 2L_k$,
$M_1,\dots, M_k$,
polynomials $q_1,\dots, q_k \in \R[x]$ where 
\begin{itemize}
\item $\deg q_j < 2L_j$
\item $q_j(0)=0$, $q_j'(0)>0$
\end{itemize}
and $\psi_1,\dots, \psi_k \in \C\{x\}$ with  
$\Im \psi_j(0)>0$
such that
\[
p(x,y) = u(x,y)\prod_{j=1}^{k} \prod_{m=1}^{M_j} (y + q_j(x) + x^{2L_j} \psi_{j}(\mu_j^m x^{1/M_{j}}))
\]
where $\mu_j = \exp(2\pi i/M_j)$.
\end{theorem}

The data $2L_1,\dots, 2L_k, M_1,\dots, M_k, q_1(x), \dots, q_k(x)$
is the most important local information about $p$.
In fact, if we define
\begin{equation} \label{pbrack}
[p](x,y) = \prod_{j=1}(y+ q_j(x) + i x^{2L_j})^{M_j}
\end{equation}
then Theorem 1.2 of \cite{BKPS} says that
\[
\frac{[p](x,y)}{p(x,y)}
\]
is bounded above and below near $0$ in $\uhp^2$.  
Thus, for many problems we can replace $p$ with the simpler polynomial $[p]$.
At the same time, Theorem \ref{purestable} makes it
possible to construct polynomials with no zeros in $\uhp^2$ or $\DD^2$
with non-trivial Puiseux structure at a boundary point---something that
was previously not known.  The construction amounts to writing out a proposed
local decomposition such as
\[
P(x,y) = (y+x +ix^2 + x^{5/2})(y+x+ix^2 - x^{5/2}) = (y+ x +ix^2) - x^5
\]
which one can show has no zeros in the $\uhp^2 \cap \DD_{\epsilon}^2$.
Then, we can find a pair of conformal maps $\phi_1(x), \phi_2(y)$ that
map $\uhp$ to a disk in $\uhp$ that is tangent to $\R$ at $0$
so that $Q = P(\phi_1(x), \phi_2(y))$ has no zeros in $\uhp^2$ yet
we do not change the fact that a local factorization of $Q$ requires
Puiseux series.  We can clear the denominators of $Q$ to get a polynomial
non-vanishing on $\uhp^2$ with non-trivial Puiseux structure at $(0,0)$.
This process blurs the original data associated with $P$, so we pose it as
a problem later (Problem \ref{problemlocal}) to show that we can construct
global stable polynomials with prescribed local data.

Theorem \ref{purestable} also led to the following conjecture and now theorem characterizing
rational functions bounded near a boundary point.

\begin{theorem}[Theorem 5.2 of \cite{BKPS}, Theorem 3 of \cite{kollar}] \label{bounded}
Assume $p$ as in Theorem \ref{purestable}.
Given $q \in \C[x,y]$, the rational function $q/p$ is bounded in
a neighborhood of $(0,0)$ intersected with $\uhp^2$ if and only if
$q$ belongs to the product ideal
\[
\prod_{j=1}^{k} (y+ q_j(x), x^{2L_j})^{M_j}
\]
in $\C\{x,y\}$.  
\end{theorem}

A different proof of this theorem was given in \cite{gKint}.  
For example, for $p=x+y-2ixy$, the numerators for locally
bounded rational functions belong to the ideal generated by $(y+x, x^2)$.

Important prior work related to Theorems \ref{purestable} and \ref{bounded}
was done in the papers \cite{BPS1}, \cite{BPS2} where the local geometry
of the zero set of $p$ was connected to the local geometry of
the level sets of rational inner functions now written in 
the upper half plane setting as $\phi(x,y) = \bar{p}(x,y)/p(x,y)$
\begin{equation} \label{level}
\{(x,y) \in \R^2: \phi(z,w) = \alpha\} \qquad \alpha \in \T.
\end{equation}
The data $2L_1,\dots, 2L_k$ from Theorem \ref{purestable}
can be used to measure the rate that the zero set of $p$
approaches $\R^2$ near $(0,0)$.
Specifically, $K_0 = \max\{2L_1,\dots, 2L_k\}$ is called the local
\hl{contact order} of $p$ at $(0,0)$ because for real $x$ near $0$
and $\epsilon>0$ sufficiently small
\[
\inf\{ |\Im y|: p(x,y) = 0, |y| < \epsilon \} \approx |x|^{K_0}.
\]
Contact order is conformally invariant and in the setting
of $\DD^2$ and $\T^2$ one can define a global contact order $K$
by taking the maximum of local contact orders at all boundary
singularities.  A main result of Bickel-Pascoe-Sola \cite{BPS1},\cite{BPS2} is
the following characterization of the integrability of derivatives of rational
inner functions.

\begin{theorem}\label{derint}
Let $\phi(z,w) = \tilde{p}(z,w)/p(z,w)$ be a rational inner function 
on $\DD^2$ with global contact order $K$.  
Then, for $1\leq \mathfrak{p} <\infty$
\[
\frac{\partial \phi}{\partial z} \in L^{\mathfrak{p}}(\T^2) 
\text{ if and only if }
\frac{\partial \phi}{\partial w} \in L^{\mathfrak{p}}(\T^2) 
\text{ if and only if }
\mathfrak{p} < 1+\frac{1}{K}.
\]
\end{theorem}

Crucial to the approach of \cite{BPS1}, \cite{BPS2}
was showing that the branches of the level sets \eqref{level}
occur in bundles called \hl{horn regions}.  Figure \ref{bundles}
shows various level sets of a rational inner function.
Note how they occur in two horn-shaped regions or bundles.
These bundles are directly connected to the
polynomials $q_1(x),\dots, q_k(x)$ in Theorem \ref{purestable}.
Necessarily, for a rational inner function $\phi$, we have $\frac{\partial \phi}{\partial z}$
belonging to $L^1(\T^2)$.  This fact even extends to several
variables.  In two variables, the paper \cite{Pascoe-cara} shows
how to construct polynomials with given global contact order
so that conditions in Theorem \ref{derint} can actually occur.

\begin{figure} 
\includegraphics[scale = .9]{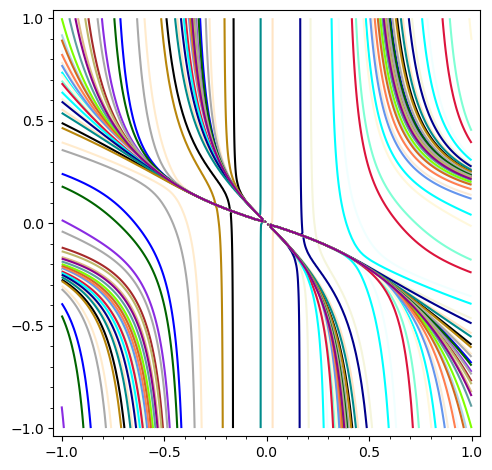}
\caption{Various level sets for a rational inner function}
\label{bundles}
\end{figure}

Finally, in the preprint \cite{gKint} we gave a complete
characterization of when a rational function $Q/p$
(where $p$ has no zero in $\uhp^2$ and no factors in common
with $\bar{p}$)
is locally in $L^{\mathfrak{p}}$ in $\R^2$ near the origin.
The main result involves checking the behavior of $Q$ along
curves related to the polynomials $q_1(x),\dots, q_k(x)$ from
Theorem \ref{purestable}.  
A few of the easiest to state offshoots are below.
In the next theorem, we consider the space of $Q/p$ that are locally square integrable:
 \[
 \mathcal{I}_p = \left\{ Q(x,y)\in \C\{x,y\}: \int_{(-\epsilon,\epsilon)^2} |Q/p|^2 dxdy <\infty \text{ for some } \epsilon >0\right\}.
 \]
 Notice that both $p$ and $\bar{p}$ belong to $\mathcal{I}_p$ so it makes sense to quotient
 $\mathcal{I}_p$ by the ideal generated by $p$, $\bar{p}$; namely $(p,\bar{p})$.

\begin{theorem}
 Assume $p$ satisfies the setup and conclusion of Theorem \ref{purestable}.
  Then,
  \[
  \dim  \mathcal{I}_p/(p,\bar{p}) = \frac{1}{2} \dim \C\{x,y\} /(p,\bar{p}).
  \]
  \end{theorem}
  
  Another easy to state result related to Theorem \ref{derint} is the following.
  
  \begin{theorem}[Theorem 1.9 of \cite{gKint}]
  Suppose $f:\DD^2 \to \C$ is bounded and rational.  Then, $\frac{\partial f}{\partial z}, \frac{\partial f}{\partial w}$
  belong to $L^1(\T^2)$.
  \end{theorem}

%It turns out that rational inner functions actually exhibit the `worst' derivative integrability behavior among
%bounded rational functions.  

Before we leave two variables, it worth highlighting a few other areas where rational inner functions
have appeared.  Recently, several authors have begun studying the \hl{Clark measures}
associated to rational inner functions \cite{Doubtsov}, \cite{Clark1}, \cite{Clark2}.  
This work involves a detailed study both
of the sums of squares decompositions of the previous section as well
as the local geometry from the present section.  
In addition, the dynamics of rational inner mappings (pairs of rational inner functions)
have been studied in \cite{dyn1}, \cite{dyn2}.  One interesting
by-product of \cite{dyn2} is that by composing rational inner functions with
singularities one can produce polynomials with high contact order.  
This leads to a question.

\begin{problem} \label{problemlocal}
Given the data $2L_1,\dots, 2L_k, M_1,\dots, M_k, q_1(x), \dots, q_k(x)$ 
from Theorem \ref{purestable}
can one always construct a polynomial $p\in \C[x,y]$ with
no zeros in $\DD^2$ associated to this data as in Theorem \ref{purestable}?
\end{problem}

\section{Higher dimensions}

In this final section we look at what is known about 
rational inner functions in three or more variables.
As indicated in Section \ref{interp}, if one could understand
interpolation for rational inner functions, it ought to 
be possible to understand interpolation more broadly
but if the answer to Problem \ref{ratinterp} is negative, then
a different approach may be necessary.

It turns out that the natural generalizations of 
Agler's Pick interpolation (Theorem \ref{aglerpick})
and the sums of squares theorem (Theorem \ref{sosthm})
do not hold for three or more variables,
and this is well-known to be connected to operator theoretic 
obstructions.  An analytic function $f:\DD^d \to \D$
satisfies a formula
\begin{equation} \label{aglerdecomp}
1-f(z) \overline{f(w)} = \sum_{j=1}^{d} (1-z_j\bar{w}_j) K_j(z,w)
\end{equation}
where the $K_j$ are positive semi-definite kernels 
if and only if $f$ satisfies an operator inequality:
for every $d$-tuple of commuting strictly
 contractive operators $T = (T_1,\dots,T_d)$
 on a Hilbert space,
 we have $\|f(T)\| \leq 1$.  
 For $d=1$, this inequality follows automatically---this is \hl{von Neumann's inequality} \cite{vN}---and
 for $d=2$ it follows automatically by \hl{And\^{o}'s inequality} \cite{Ando}.  
For $d>2$, it does not follow automatically because there are counterexamples (first discovered
by Varopoulos \cite{varo}).
For this reason it makes sense to study the class of functions for which the operator
inequality holds.

\begin{definition}
We say analytic $f: \DD^d \to \DD$ belongs to the \hl{\emph{Schur-Agler class} }
(or Agler class for short)
if for every $d$-tuple of commuting strictly
 contractive operators $T = (T_1,\dots,T_d)$
 we have $\|f(T)\| \leq 1$. 
 \end{definition}
 
In particular, we can study rational inner functions in the Schur-Agler class.

\begin{theorem}
A rational inner function $\phi(z) = \tilde{p}(z)/p(z)$ on $\DD^d$
belongs to the Schur-Agler class if and only if we have a formula
\[
|p(z)|^2 - |\tilde{p}(z)|^2 = \sum_{j=1}^{d} (1-|z_j|^2) \mathcal{E}_j(z)
\]
where each $\mathcal{E}_j(z)$ is a sum of squared moduli of polynomials.
\end{theorem}

The concept of this theorem is essentially due to Cole-Wermer \cite{CW99},
while it is explicitly stated in \cite{gKratag} with explicit degree bounds for
the terms in the sums of squares terms.  
Certain low-degree or regular rational inner functions automatically belong
to the Agler class. 
\begin{itemize}
\item If $p(z_1,z_2,z_3)$ with no zeros in $\DD^3$ 
has multidegree $(n,1,1)$ then $\tilde{p}/p$ is in the Agler class;
this is proven in \cite{gKtridisk} with some elaborations in \cite{BK}.
\item If $p(z_1,z_2,z_3)$ has multidegree $(n,m,1)$
and no zeros in $\DD^3$,
then there is a monomial $z^{\alpha}$ such that $z^{\alpha} \tilde{p}/p$
belongs to the Agler class. See \cite{gKtridisk}.
(Note the theorem in \cite{gKtridisk} states $p$ has no
zeros in $\cd^3$ but this is not essential if we use a theorem
of Scheiderer \cite{scheid} on non-negative trigonometric polynomials.)
\item In \cite{drexelcon}, it is proven that if $p\in \C[z_1,\dots,z_d]$
has no zeros on $\cd^d$ then there is a monomial $z^{\alpha}$
such that $z^{\alpha} \tilde{p}/p$
belongs to the Agler class.
\item In the paper \cite{gKsym} we gave a concrete checkable
condition for when polynomials $p(z_1,\dots, z_d)$ with no zeros 
in $\DD^d$ which are multiaffine (meaning they have degree
one in every variable) and are symmetric give rise
to $\phi = \tilde{p}/p$ belonging to the Agler class.
The following is not known as far as we know.
\end{itemize}

\begin{problem}
Suppose $p(z_1,\dots, z_d)$ has no zeros 
in $\DD^d$, is multiaffine and symmetric.
Does $\phi = \tilde{p}/p$ automatically belong to the Agler class?
\end{problem}

The answer is probably no, but if true it would represent
a novel strengthening of the \hl{Grace-Walsh-Szeg\H{o} theorem}.

One can define an \hl{Agler norm} for analytic $f:\DD^d \to \C$ via
\[
\|f\|_{A} = \sup_{T} \|f(T)\|
\]
where the supremum is taken over commuting strictly contractive $d$-tuples.
We could have $\|f\|_{A} = \infty$.
A phenomenon noted in \cite{drexelschwarz}
is that there exists a rational inner function $\phi$ in three variables
with $\|\phi\|_{A}>1$ and there exists a monomial $z^{\alpha}$ such that
$\|z^{\alpha} \phi\|_{A} = 1$.   

A basic open question suggested by V. Vinnikov is the following:
\begin{problem} \label{ratagler}
Does every rational inner function have finite Agler norm?
\end{problem}

By absolute convergence properties, if $p$ has no zeros on $\cd^d$ then
necessarily $\tilde{p}/p$ has finite Agler norm so
the question is about rational inner functions with boundary singularities.
This leads to further motivation for studying boundary singularities of rational functions.

The paper \cite{BPS3} presents a detailed study of 
the boundary level sets of rational inner functions in higher dimensions.
One of the key results demonstrates that integrability of the derivative 
of a rational inner function depends on the rate that
the zero set of $p$ approaches the boundary.
This paper also presents a variety of interesting examples
where the zero set of $p$ intersected with $\T^3$ has either
an isolated zero or a curve of singularities.  There seems to 
be no obvious connection between coarse boundary nature and
derivative integrability.
The paper \cite{Bergqvist} has some further results
on integrability in higher dimensions.
A basic question coming out of \cite{BPS3} is the following.
\begin{problem}
The exact derivative integrability cut-offs are known for rational inner functions
in two variables by Theorem \ref{derint}.  What are the possible integrability
cut-offs in three or more variables?
Specifically, given a rational inner function $\phi$ on $\DD^d$ for $d>2$,
what are all of the possible values of
\[
\sup\left\{ \mathfrak{p}>0 : \frac{\partial \phi}{\partial z_1} \in L^{\mathfrak{p}}(\T^d)\right\} \text{ ?}
\]
\end{problem}

Finally, the preprint \cite{BKPS2} studies the problem of characterizing
$q/p$ in three variables that are bounded near a boundary singularity in
the simplest case, namely when the zero set of $p$ is smooth at
the boundary singularity.  There remain many open questions in
three or more variables.

\section{Acknowledgements}
I would like to thank Alan Sola for offering comments on
a draft of this paper (on extremely short notice!). Also, thanks
to all the collaborators and friends who have contributed
to the interesting (to me) topic of rational inner functions.
Finally, thank you to the anonymous referee for numerous
helpful suggestions.
An additional thanks goes to Rongwei Yang for pointing
out a typo in the proof of Prop 2.

%Test \cite{Pfister}

 \bibliography{RIFbib}{}

\begin{thebibliography}{10}
\providecommand{\url}[1]{{#1}}
\providecommand{\urlprefix}{URL }
\expandafter\ifx\csname urlstyle\endcsname\relax
  \providecommand{\doi}[1]{DOI~\discretionary{}{}{}#1}\else
  \providecommand{\doi}{DOI~\discretionary{}{}{}\begingroup
  \urlstyle{rm}\Url}\fi

\bibitem{Abate}
Abate, M.: The {J}ulia-{W}olff-{C}arath\'eodory theorem in polydisks.
\newblock J. Anal. Math. \textbf{74}, 275--306 (1998).
\newblock \doi{10.1007/BF02819453}.
\newblock \urlprefix\url{https://doi.org/10.1007/BF02819453}

\bibitem{Agler1}
Agler, J.: On the representation of certain holomorphic functions defined on a
  polydisc.
\newblock In: Topics in operator theory: {E}rnst {D}. {H}ellinger memorial
  volume, \emph{Oper. Theory Adv. Appl.}, vol.~48, pp. 47--66. Birkh\"auser,
  Basel (1990)

\bibitem{AM99}
Agler, J., McCarthy, J.E.: Nevanlinna-{P}ick interpolation on the bidisk.
\newblock J. Reine Angew. Math. \textbf{506}, 191--204 (1999).
\newblock \doi{10.1515/crll.1999.004}.
\newblock \urlprefix\url{https://doi.org/10.1515/crll.1999.004}

\bibitem{PickBook}
Agler, J., McCarthy, J.E.: Pick interpolation and {H}ilbert function spaces,
  \emph{Graduate Studies in Mathematics}, vol.~44.
\newblock American Mathematical Society, Providence, RI (2002).
\newblock \doi{10.1090/gsm/044}.
\newblock \urlprefix\url{https://doi.org/10.1090/gsm/044}

\bibitem{AMYbook}
Agler, J., McCarthy, J.E., Young, N.: Operator analysis---{H}ilbert space
  methods in complex analysis, \emph{Cambridge Tracts in Mathematics}, vol.
  219.
\newblock Cambridge University Press, Cambridge (2020).
\newblock \doi{10.1017/9781108751292}.
\newblock \urlprefix\url{https://doi.org/10.1017/9781108751292}

\bibitem{AMYcara}
Agler, J., McCarthy, J.E., Young, N.J.: A {C}arath\'eodory theorem for the
  bidisk via {H}ilbert space methods.
\newblock Math. Ann. \textbf{352}(3), 581--624 (2012).
\newblock \doi{10.1007/s00208-011-0650-7}.
\newblock \urlprefix\url{https://doi.org/10.1007/s00208-011-0650-7}

\bibitem{AMY}
Agler, J., McCarthy, J.E., Young, N.J.: Operator monotone functions and
  {L}\"owner functions of several variables.
\newblock Ann. of Math. (2) \textbf{176}(3), 1783--1826 (2012).
\newblock \doi{10.4007/annals.2012.176.3.7}.
\newblock \urlprefix\url{https://doi.org/10.4007/annals.2012.176.3.7}

\bibitem{ATYboundary}
Agler, J., Tully-Doyle, R., Young, N.J.: Boundary behavior of analytic
  functions of two variables via generalized models.
\newblock Indag. Math. (N.S.) \textbf{23}(4), 995--1027 (2012).
\newblock \doi{10.1016/j.indag.2012.07.003}.
\newblock \urlprefix\url{https://doi.org/10.1016/j.indag.2012.07.003}

\bibitem{Clark2}
Anderson, J.T., Bergqvist, L., Bickel, K., Cima, J.A., Sola, A.A.: Clark
  measures for rational inner functions {II}: general bidegrees and higher
  dimensions (2023).
\newblock \urlprefix\url{https://arxiv.org/abs/2303.11248}.
\newblock To appear in Ark. Mat.

\bibitem{Ando}
And\^o, T.: On a pair of commutative contractions.
\newblock Acta Sci. Math. (Szeged) \textbf{24}, 88--90 (1963)

\bibitem{ABG}
Atiyah, M.F., Bott, R., G\.arding, L.: Lacunas for hyperbolic differential
  operators with constant coefficients. {I}.
\newblock Acta Math. \textbf{124}, 109--189 (1970).
\newblock \doi{10.1007/BF02394570}.
\newblock \urlprefix\url{https://doi.org/10.1007/BF02394570}

\bibitem{BSV}
Ball, J.A., Sadosky, C., Vinnikov, V.: Scattering systems with several
  evolutions and multidimensional input/state/output systems.
\newblock Integral Equations Operator Theory \textbf{52}(3), 323--393 (2005).
\newblock \doi{10.1007/s00020-005-1351-y}.
\newblock \urlprefix\url{https://doi.org/10.1007/s00020-005-1351-y}

\bibitem{Bergqvist}
Bergqvist, L.: Rational inner functions and their {D}irichlet type norms.
\newblock Comput. Methods Funct. Theory \textbf{23}(3), 563--587 (2023).
\newblock \doi{10.1007/s40315-022-00465-1}.
\newblock \urlprefix\url{https://doi.org/10.1007/s40315-022-00465-1}

\bibitem{Clark1}
Bickel, K., Cima, J.A., Sola, A.A.: Clark measures for rational inner
  functions.
\newblock Michigan Math. J. \textbf{73}(5), 1021--1057 (2023).
\newblock \doi{10.1307/mmj/20216046}.
\newblock \urlprefix\url{https://doi.org/10.1307/mmj/20216046}

\bibitem{BK}
Bickel, K., Knese, G.: Inner functions on the bidisk and associated {H}ilbert
  spaces.
\newblock J. Funct. Anal. \textbf{265}(11), 2753--2790 (2013).
\newblock \doi{10.1016/j.jfa.2013.08.002}.
\newblock \urlprefix\url{https://doi.org/10.1016/j.jfa.2013.08.002}

\bibitem{BKPS}
Bickel, K., Knese, G., Pascoe, J.E., Sola, A.: Local theory of stable
  polynomials and bounded rational functions of several variables (2024).
\newblock \urlprefix\url{https://arxiv.org/abs/2109.07507}.
\newblock To appear in Ann. Polon. Math.

\bibitem{BKPS2}
Bickel, K., Knese, G., Pascoe, J.E., Sola, A.: Stable polynomials and
  admissible numerators in product domains (2024).
\newblock \urlprefix\url{https://arxiv.org/abs/2406.13014}

\bibitem{BPS1}
Bickel, K., Pascoe, J.E., Sola, A.: Derivatives of rational inner functions:
  geometry of singularities and integrability at the boundary.
\newblock Proc. Lond. Math. Soc. (3) \textbf{116}(2), 281--329 (2018).
\newblock \doi{10.1112/plms.12072}.
\newblock \urlprefix\url{https://doi.org/10.1112/plms.12072}

\bibitem{BPS2}
Bickel, K., Pascoe, J.E., Sola, A.: Level curve portraits of rational inner
  functions.
\newblock Ann. Sc. Norm. Super. Pisa Cl. Sci. (5) \textbf{21}, 449--494 (2020)

\bibitem{BPS3}
Bickel, K., Pascoe, J.E., Sola, A.: Singularities of rational inner functions
  in higher dimensions.
\newblock Amer. J. Math. \textbf{144}(4), 1115--1157 (2022).
\newblock \doi{10.1353/ajm.2022.0025}.
\newblock \urlprefix\url{https://doi.org/10.1353/ajm.2022.0025}

\bibitem{CW99}
Cole, B.J., Wermer, J.: Ando's theorem and sums of squares.
\newblock Indiana Univ. Math. J. \textbf{48}(3), 767--791 (1999).
\newblock \doi{10.1512/iumj.1999.48.1716}.
\newblock \urlprefix\url{https://doi.org/10.1512/iumj.1999.48.1716}

\bibitem{Doubtsov}
Doubtsov, E.: Clark measures on the torus.
\newblock Proc. Amer. Math. Soc. \textbf{148}(5), 2009--2017 (2020).
\newblock \doi{10.1090/proc/14846}.
\newblock \urlprefix\url{https://doi.org/10.1090/proc/14846}

\bibitem{Fulton}
Fulton, W.: Algebraic curves.
\newblock Advanced Book Classics. Addison-Wesley Publishing Company, Advanced
  Book Program, Redwood City, CA (1989).
\newblock An introduction to algebraic geometry, Notes written with the
  collaboration of Richard Weiss, Reprint of 1969 original

\bibitem{BBook}
Garcia, S.R., Mashreghi, J., Ross, W.T.: Finite {B}laschke products and their
  connections.
\newblock Springer, Cham (2018).
\newblock \doi{10.1007/978-3-319-78247-8}.
\newblock \urlprefix\url{https://doi.org/10.1007/978-3-319-78247-8}

\bibitem{GW}
Geronimo, J.S., Woerdeman, H.J.: Positive extensions, {F}ej\'er-{R}iesz
  factorization and autoregressive filters in two variables.
\newblock Ann. of Math. (2) \textbf{160}(3), 839--906 (2004).
\newblock \doi{10.4007/annals.2004.160.839}.
\newblock \urlprefix\url{https://doi.org/10.4007/annals.2004.160.839}

\bibitem{drexelcon}
Grinshpan, A., Kaliuzhnyi-Verbovetskyi, D.S., Vinnikov, V., Woerdeman, H.J.:
  Contractive determinantal representations of stable polynomials on a matrix
  polyball.
\newblock Math. Z. \textbf{283}(1-2), 25--37 (2016).
\newblock \doi{10.1007/s00209-015-1587-4}.
\newblock \urlprefix\url{https://doi.org/10.1007/s00209-015-1587-4}

\bibitem{DrexelStable}
Grinshpan, A., Kaliuzhnyi-Verbovetskyi, D.S., Vinnikov, V., Woerdeman, H.J.:
  Stable and real-zero polynomials in two variables.
\newblock Multidimens. Syst. Signal Process. \textbf{27}(1), 1--26 (2016).
\newblock \doi{10.1007/s11045-014-0286-3}.
\newblock \urlprefix\url{https://doi.org/10.1007/s11045-014-0286-3}

\bibitem{drexelschwarz}
Grinshpan, A., Kaliuzhnyi-Verbovetskyi, D.S., Woerdeman, H.J.: The {S}chwarz
  lemma and the {S}chur-{A}gler class.
\newblock In: {M}athematical {T}heory of {N}etworks and {S}ystems July 7-11,
  2014. {G}roningen, {T}he {N}etherlands, pp. 1835--1836 (2014)

\bibitem{gKBS}
Knese, G.: Bernstein-{S}zeg{\H{o}} measures on the two dimensional torus.
\newblock Indiana Univ. Math. J. \textbf{57}(3), 1353--1376 (2008).
\newblock \doi{10.1512/iumj.2008.57.3226}.
\newblock \urlprefix\url{https://doi.org/10.1512/iumj.2008.57.3226}

\bibitem{gKnoz}
Knese, G.: Polynomials with no zeros on the bidisk.
\newblock Anal. PDE \textbf{3}(2), 109--149 (2010).
\newblock \doi{10.2140/apde.2010.3.109}.
\newblock \urlprefix\url{https://doi.org/10.2140/apde.2010.3.109}

\bibitem{gKratag}
Knese, G.: Rational inner functions in the {S}chur-{A}gler class of the
  polydisk.
\newblock Publ. Mat. \textbf{55}(2), 343--357 (2011).
\newblock \doi{10.5565/PUBLMAT\_55211\_04}.
\newblock \urlprefix\url{https://doi.org/10.5565/PUBLMAT_55211_04}

\bibitem{gKtridisk}
Knese, G.: Schur-{A}gler class rational inner functions on the tridisk.
\newblock Proc. Amer. Math. Soc. \textbf{139}(11), 4063--4072 (2011).
\newblock \doi{10.1090/S0002-9939-2011-10975-4}.
\newblock \urlprefix\url{https://doi.org/10.1090/S0002-9939-2011-10975-4}

\bibitem{gKsym}
Knese, G.: Stable symmetric polynomials and the {S}chur-{A}gler class.
\newblock Illinois J. Math. \textbf{55}(4), 1603--1620 (2011).
\newblock \urlprefix\url{http://projecteuclid.org/euclid.ijm/1373636698}

\bibitem{gKintreg}
Knese, G.: Integrability and regularity of rational functions.
\newblock Proc. Lond. Math. Soc. (3) \textbf{111}(6), 1261--1306 (2015).
\newblock \doi{10.1112/plms/pdv061}.
\newblock \urlprefix\url{https://doi.org/10.1112/plms/pdv061}

\bibitem{gK3x3}
Knese, G.: The von {N}eumann inequality for {$3\times 3$} matrices.
\newblock Bull. Lond. Math. Soc. \textbf{48}(1), 53--57 (2016).
\newblock \doi{10.1112/blms/bdv087}.
\newblock \urlprefix\url{https://doi.org/10.1112/blms/bdv087}

\bibitem{gKextreme}
Knese, G.: Extreme points and saturated polynomials.
\newblock Illinois J. Math. \textbf{63}(1), 47--74 (2019).
\newblock \doi{10.1215/00192082-7600059}.
\newblock \urlprefix\url{https://doi.org/10.1215/00192082-7600059}

\bibitem{gKKummert}
Knese, G.: Kummert's approach to realization on the bidisk.
\newblock Indiana Univ. Math. J. \textbf{70}(6), 2369--2403 (2021).
\newblock \doi{10.1512/iumj.2021.70.8738}.
\newblock \urlprefix\url{https://doi.org/10.1512/iumj.2021.70.8738}

\bibitem{gKint}
Knese, G.: Boundary local integrability of rational functions in two variables
  (2024).
\newblock \urlprefix\url{https://arxiv.org/abs/2404.05042}

\bibitem{kollar}
Koll\'ar, J.: Bounded meromorphic functions on the complex 2-disc.
\newblock Period. Math. Hungar. \textbf{88}(1), 1--7 (2024).
\newblock \doi{10.1007/s10998-023-00538-1}.
\newblock \urlprefix\url{https://doi.org/10.1007/s10998-023-00538-1}

\bibitem{Kosinski}
Kosi\'nski, L.: Three-point {N}evanlinna-{P}ick problem in the polydisc.
\newblock Proc. Lond. Math. Soc. (3) \textbf{111}(4), 887--910 (2015).
\newblock \doi{10.1112/plms/pdv045}.
\newblock \urlprefix\url{https://doi.org/10.1112/plms/pdv045}

\bibitem{Kummert}
Kummert, A.: Synthesis of two-dimensional lossless {$m$}-ports with prescribed
  scattering matrix.
\newblock Circuits Systems Signal Process. \textbf{8}(1), 97--119 (1989).
\newblock \doi{10.1007/BF01598747}.
\newblock \urlprefix\url{https://doi.org/10.1007/BF01598747}

\bibitem{KummertStable}
Kummert, A.: 2-{D} stable polynomials with parameter-dependent coefficients:
  generalizations and new results.
\newblock pp. 725--731 (2002).
\newblock \doi{10.1109/TCSI.2002.1010028}.
\newblock \urlprefix\url{https://doi.org/10.1109/TCSI.2002.1010028}.
\newblock Special issue on multidimensional signals and systems

\bibitem{Landau}
Landau, H.J.: Maximum entropy and the moment problem.
\newblock Bull. Amer. Math. Soc. (N.S.) \textbf{16}(1), 47--77 (1987).
\newblock \doi{10.1090/S0273-0979-1987-15464-4}.
\newblock \urlprefix\url{https://doi.org/10.1090/S0273-0979-1987-15464-4}

\bibitem{MPcara}
McCarthy, J.E., Pascoe, J.E.: The {J}ulia-{C}arath\'eodory theorem on the
  bidisk revisited.
\newblock Acta Sci. Math. (Szeged) \textbf{83}(1-2), 165--175 (2017).
\newblock \doi{10.14232/actasm-016-311-x}.
\newblock \urlprefix\url{https://doi.org/10.14232/actasm-016-311-x}

\bibitem{vN}
von Neumann, J.: Eine {S}pektraltheorie f\"ur allgemeine {O}peratoren eines
  unit\"aren {R}aumes.
\newblock Math. Nachr. \textbf{4}, 258--281 (1951).
\newblock \doi{10.1002/mana.3210040124}.
\newblock \urlprefix\url{https://doi.org/10.1002/mana.3210040124}

\bibitem{Pascoe-cara}
Pascoe, J.E.: An inductive {J}ulia-{C}arath\'eodory theorem for {P}ick
  functions in two variables.
\newblock Proc. Edinb. Math. Soc. (2) \textbf{61}(3), 647--660 (2018).
\newblock \doi{10.1017/s0013091517000396}.
\newblock \urlprefix\url{https://doi.org/10.1017/s0013091517000396}

\bibitem{Pfister}
Pfister, A.: \"uber das {K}oeffizientenproblem der beschr\"ankten {F}unktionen
  von zwei {V}er\"anderlichen.
\newblock Math. Ann. \textbf{146}, 249--262 (1962).
\newblock \doi{10.1007/BF01470954}.
\newblock \urlprefix\url{https://doi.org/10.1007/BF01470954}

\bibitem{RudinBook}
Rudin, W.: Function theory in polydiscs.
\newblock W. A. Benjamin, Inc., New York-Amsterdam (1969)

\bibitem{scheid}
Scheiderer, C.: Sums of squares on real algebraic surfaces.
\newblock Manuscripta Math. \textbf{119}(4), 395--410 (2006).
\newblock \doi{10.1007/s00229-006-0630-5}.
\newblock \urlprefix\url{https://doi.org/10.1007/s00229-006-0630-5}

\bibitem{simonbasic}
Simon, B.: Basic complex analysis, \emph{A Comprehensive Course in Analysis},
  vol. Part 2A.
\newblock American Mathematical Society, Providence, RI (2015).
\newblock \doi{10.1090/simon/002.1}.
\newblock \urlprefix\url{https://doi.org/10.1090/simon/002.1}

\bibitem{dyn2}
Sola, A.: A note on polydegree {$(n,1)$} rational inner functions, slice
  matrices, and singularities.
\newblock Arch. Math. (Basel) \textbf{120}(2), 171--181 (2023).
\newblock \doi{10.1007/s00013-022-01812-3}.
\newblock \urlprefix\url{https://doi.org/10.1007/s00013-022-01812-3}

\bibitem{dyn1}
Sola, A., Tully-Doyle, R.: Dynamics of low-degree rational inner skew-products
  on {$\mathbb{T}^2$}.
\newblock Ann. Polon. Math. \textbf{128}(3), 249--273 (2022).
\newblock \doi{10.4064/ap211108-28-2}.
\newblock \urlprefix\url{https://doi.org/10.4064/ap211108-28-2}

\bibitem{Tully-Doyle}
Tully-Doyle, R.: Analytic functions on the bidisk at boundary singularities via
  {H}ilbert space methods.
\newblock Oper. Matrices \textbf{11}(1), 55--70 (2017).
\newblock \doi{10.7153/oam-11-04}.
\newblock \urlprefix\url{https://doi.org/10.7153/oam-11-04}

\bibitem{varo}
Varopoulos, N.T.: On an inequality of von {N}eumann and an application of the
  metric theory of tensor products to operators theory.
\newblock J. Functional Analysis \textbf{16}, 83--100 (1974).
\newblock \doi{10.1016/0022-1236(74)90071-8}.
\newblock \urlprefix\url{https://doi.org/10.1016/0022-1236(74)90071-8}

\end{thebibliography}
 \bibliographystyle{spmpsci}
 
% \begin{bibdiv}\begin{biblist}
  %\end{biblist}\end{bibdiv}

\end{document}